\documentclass{amsart}

\usepackage{amsfonts,amssymb,amsmath,amsthm,amscd}
\usepackage[vcentermath]{youngtab}
\usepackage{multirow}

\usepackage{url}
\usepackage{enumerate}

\newtheorem{thm}{Theorem}[section]
\newtheorem{prop}[thm]{Proposition}                          
\newtheorem{claim}[thm]{Claim}

\newtheorem{lem}[thm]{Lemma}
\theoremstyle{definition}
\newtheorem{defn}[thm]{Definition}
\newtheorem{rem}[thm]{Remark}
\newtheorem{exa}[thm]{Example}

\newcommand{\gl}{\operatorname{GL}}
\newcommand{\gL}{\mathfrak{gl}}
\newcommand{\df}{\displaystyle\frac}
\newcommand{\seq}[1]{\left<#1\right>}   
\newcommand{\F}{\mathcal{F}}
\newcommand{\End}{\operatorname{End}}
\newcommand{\T}{\mathcal{T}}
\newcommand{\Ts}{\mathcal{T}^{\mu}}
\newcommand{\MM}[2]{M^{#1,#2}(\mathbb{Z}_{\geq 0})}
\newcommand{\card}{\operatorname{card}}
\newcommand{\std}{\operatorname{std}}
\newcommand{\im}{\operatorname{Im}}
\newcommand{\N}{\mathcal{N}}
\newcommand{\PP}{\mathfrak{P}}

\newcommand{\Id}{\operatorname{Id}}
\newcommand{\Xmm}{X^{\mu,\mu'}}
\newcommand{\Ymm}{Y^{\mu,\mu'}}
\newcommand{\TT}{\textbf{T}}
\newcommand{\D}{\mathfrak{D}}
\newcommand{\Zt}{Z^{\textbf{t}}}
\newcommand{\tlo}{\tilde{\omega}}
\newcommand{\rk}{\operatorname{rank}}
\newcommand{\sll}{\mathfrak{sl}}
\newcommand{\usl}[1]{\mathcal{U}(\sll_{#1})}

\newcommand{\yd}{2d}

\newcommand{\yy}{10}
\newcommand{\yya}{11}
\newcommand{\yyb}{12}
\newcommand{\yyc}{13}
\newcommand{\yyd}{14}

\author{Daniele Rosso}
\address{\newline
Daniele Rosso \newline 
The University of Chicago \newline 
Department of Mathematics \newline
5734 S. University Ave. Chicago, IL 60637}

\email{d\_rosso@math.uchicago.edu}

\begin{document}

\title[Classic and mirabolic RSK correspondence for partial flags]{Classic and mirabolic Robinson-Schensted-Knuth correspondence for partial flags}

\date{\today}
\begin{abstract}In this paper we first generalize to the case of partial flags a result proved both by Spaltenstein and by Steinberg that relates the relative position of two complete flags and the irreducible components of the flag variety in which they lie, using the Robinson-Schensted-Knuth correspondence. Then we use this result to generalize the mirabolic Robinson-Schensted-Knuth correspondence defined by Travkin, to the case of two partial flags and a line.\end{abstract} 
\maketitle

\section{Introduction}
\subsection{}
The Robinson-Schensted-Knuth correspondence (RSK for short) is a very classical result. It was first discovered by Robinson (see \cite{R}) as a bijection between permutations of $d$ letters and pairs of standard Young tableaux of the same shape on $d$ boxes, then independently rediscovered by Schensted (see \cite{Sc}). It was eventually generalized by Knuth (see \cite{K}) to the case of two rowed arrays in lexicographic order (or equivalenty matrices with nonnegative integer entries) and pairs of semistandard Young tableaux of the same shape.

This  correspondence comes up when considering flag varieties. 
The Bruhat decomposition tells us that the relative position of two complete flags in a $d$-dimensional space $V$
is given by an element of the symmetric group $S_d$. 
Also, given a nilpotent $x\in\End(V)$, the irreducible components of the subvariety of flags that are preserved by $x$ are parametrized by the standard tableaux on the shape $\lambda$, which is the Jordan type of $x$ (see \cite[II 5.21]{Sp}, \cite{St}).
Then it is a theorem (see \cite[II 9.8]{Sp}, \cite{St}) that, for two general flags, their relative position is given by the permutation that we get applying the RSK correspondence to the standard tableaux associated to the irreducible components in which they lie.

\subsection{}
We would like to generalize the above to pairs of partial flags.

For a nilpotent transformation $x$, we consider the following variety of $n$-step partial flags that are preserved by $x$:
$$\{F:\quad 0=F_0\subset F_1\subset\ldots\subset F_{n-1}\subset F_n=V |x(F_i)\subset F_{i-1} \quad\forall i\}.$$
The irreducible components of this variety can be parametrized by `semistandard' tableaux (better, by transposed of semistandard tableaux, more on this later) by applying to our specific case some results of Haines about the fibers of convolution morphisms in the affine Grassmanian (see \cite{H}).
This parametrization is also essentially the same that Spaltenstein shows in \cite{Sp1}.

Notice that Shimomura has also worked on partial flag varieties and in \cite{Sh} has given a parametrization of the irreducible components of the variety of partial flags that are invariant under a nilpotent transformation, using Young tableaux, but the variety he considers is different from ours. 


Given two flags $F$, $F'$ (partial or complete) we define the \emph{relative position} of $F$ and $F'$ to be the matrix of nonnegative integers $M(F,F')$ with entries given by:
\begin{equation}\label{rpos}M(F,F')_{ij}=\dim\left(\frac{F_i\cap F'_j}{F_i\cap F'_{j-1}+F_{i-1}\cap F'_j}\right).\end{equation}

Then, see \cite[1.1]{blm}, the set of such matrices parametrizes the orbits of the  diagonal action of $\gl_d$ on the set of pairs of flags. 

It seems then natural to ask if the theorem generalizes to the case of partial flags. Given two partial flags, is the matrix of relative position the one that corresponds through the more general RSK correspondence to the two semistandard tableaux indexing the irreducible components in which the flags lie?

As we prove in Theorem \ref{rosso}, the answer is yes, if we modify slightly the usual conventions for the RSK correspondence. We need a variation to account for the fact that the `semistandard' tableaux mentioned earlier are actually transposed of semistandard tableaux (i.e. the strictness of the inequalities is switched from rows to columns and viceversa).

\subsection{}
The second part of the paper is concerned with generalizing Travkin's construction from \cite{T} to the case of partial flags and not just complete flags. We generalize his algorithm and then, using the results of the first part, we show that the generalization agrees with the geometry of the varieties involved.

The diagonal action of $\gl(V)$ on the variety of triples of two flags and a line has orbits that can be parametrized by pairs $(M,\Delta)$ (see \cite{MWZ}). Here $M$ is the relative position of the two flags, as in \eqref{rpos}, and $\Delta$ is some more combinatorial data (which we will see more precisely in Section \ref{sec2}) that tells us where the line lies. In the case where the flags are complete, the matrix $M$ is just a permutation matrix.

If we only consider complete flags, then, the set parametrizing the orbits can be thought of as the set of colored permutations $RB$, that is permutation words where every letter is assigned one of two colors (say red and blue).

In his paper \cite{T}, Travkin has introduced the \emph{mirabolic Robinson-Schensted-Knuth correspondence}. It is a bijection between $RB$ and the set of all $\{(\lambda, \theta, \lambda', T, T')\}$, where $T$, $T'$ are standard Young tableaux of shape $\lambda$ and $\lambda'$ respectively, and $\theta$ is another partition that satisfies $\lambda_i\geq\theta_i\geq\lambda_{i-1}$ and $\lambda'_i\geq\theta_i\geq\lambda'_{i-1}$ for all $i$.
This mirabolic RSK correspondence has a geometric meaning: given a colored permutation indexing a $\gl(V)$-orbit on the space of two complete flags and a line, it describes the type of a generic conormal vector to the orbit.

Many arguments in the second part of the paper are just adaptations of Travkin's arguments to the case of partial flags.

\subsection{}
This paper is part of an ongoing project that studies the convolution algebras of $\gl(V)$-equivariant functions on varieties of triples of two $n$-step partial flags and a line. We have partial results for the cases $n=2,3$ where we get a direct summand isomorphic to $M_n(\usl{n})$. These involve finding a rather complicated central element in the algebra. We believe that the mirabolic RSK correspondence for partial flags will help us find central elements and hopefully generalize these results to any $n$.

\section{Flag Varieties and Tableaux}\label{notat}
Let us fix some notation. 

For any set $X$, we will denote its cardinality by $\card X$. 

We denote by $S_d$ the symmetric group on $d$ elements.

We let $V$ be a $d$-dimensional vector space over the field $k$, and $\F$ be the variety of complete flags in $V$. 

We let $G$ be the general linear group $G=\gl(V)\simeq\gl_d$ and we let $\N$ be the set of nilpotent elements in $\End(V)$.
If $x\in\N$, we let its Jordan type be $\lambda=(\lambda_1,\lambda_2,\ldots,\lambda_m)$. Then $\lambda$ is a partition of $d$, which means that it satisfies $\lambda_1\geq\lambda_2\geq\ldots\lambda_m$, and $|\lambda|=\lambda_1+\lambda_2+\ldots+\lambda_m=d$. 

We consider the subvariety $\F_x\subset\F$ of flags preserved by $x$, that is
$$\F_x:=\{F\in\F|x(F_i)\subset F_{i-1}\}.$$
\begin{defn}\label{deft}
Now let $\T_\lambda$ be the set of standard Young tableaux of shape $\lambda$, we can define a map
$$t:\F_x\to\T_\lambda$$
in the following way: given $F\in\F_x$, consider the Jordan type of the restriction $x|_{F_i}$. This gives us an increasing sequence of Young diagrams each with one box more than the previous one. Filling the new box with the number $i$ at each step, we get a standard tableau.
\end{defn}
Then (see \cite[II 5.21]{Sp},\cite{St}) for a tableau $T\in\T_\lambda$, if we let $\F_{x,T}=t^{-1}(T)\subset\F_x$, we have that the closure $C_{x,T}=\overline{\F_{x,T}}$ is an irreducible component of $\F_x$. All the irreducible components are parametrized in this way by the set of standard tableaux of shape $\lambda$. In \cite{Sp}, Spaltenstein actually uses a slightly different parametrization, to see how the two parametrizations are related, see \cite{vL}.
\begin{defn}In this paper, whenever we will refer to a \emph{general} element in a variety or subvariety, we will mean any element in a suitable open dense subset. \end{defn}


We can now state the result ( \cite[II 9.8]{Sp} and \cite[1.1]{St}) that we wish to generalize in the first part of this paper.

\begin{thm}\label{stein}Let $\F$ be the variety of complete flags on a vector space $V$, and $x\in\End(V)$ a nilpotent transformation of Jordan type $\lambda$. Let $T,T'$ be standard Young tableaux of shape $\lambda$ and $C_{x,T}$ and $C_{x,T'}$ the corresponding irreducible components of $\F_x$. Then for general flags $F\in C_{x,T}$ and $F'\in C_{x,T'}$, the permutation $w(F,F')$ that gives the relative position of the two flags is the same as the permutation $w(T,T')$ given by the RSK correspondence.
\end{thm}

Our goal is to extend this result to varieties of partial flags.

\subsection{Partial Flags and Semistandard Tableaux}\label{pfsst}

Let us fix an integer $n\geq1$ and let $\mu$ be a \emph{composition} of $d$, that is $\mu=(\mu_1,\ldots,\mu_n)$ a sequence of positive integers, such that $|\mu|=\mu_1+\mu_2+\ldots+\mu_n=d$ ($\mu$ is not necessarily a partition because we do not require it to be decreasing). We have the variety of $n$-step flags of type $\mu$ in $V$
$$\F^\mu:=\{F=(0=F_0\subset F_1\subset\ldots\subset F_{n-1}\subset F_n=V)|\dim(F_i/F_{i-1})=\mu_i\}.$$
Then for $x$ as before, we consider the subvariety of partial flags that are preserved by $x$:
$$\F^\mu_x:=\{F\in\F^\mu|x(F_i)\subset F_{i-1}\}.$$
If $F\in\F^\mu_x$, we can associate to $F$ a tableau in an analogous way to definition \ref{deft}, except this time at each step we are adding several boxes, none of which will be in the same row. 
The result will be a tableau which is strictly increasing along rows and weakly increasing down columns. For the purpose of this discussion, we will call this kind of tableaux \emph{semistandard}, although by the usual definition this is the transposed of a semistandard tableau.

\begin{defn}Given any tableau $T$ with entries in $\{1,\ldots,n\}$, we say that its \emph{content} is the sequence $\mu=\mu(T)=(\mu_1,\ldots,\mu_n)$ where $\mu_i$ is the number of times the entry $i$ appears in $T$.\end{defn}
 
\begin{defn}\label{eqdeft}
So, if we let $\Ts_\lambda$ be the set of semistandard tableaux of shape $\lambda$ and content $\mu$, we just defined a map
$$t:\F_x^\mu\longrightarrow \Ts_\lambda.$$
\end{defn}

\begin{lem}The irreducible components of $\F_x^\mu$ are the closures $C_{x,T}=\overline{\F_{x,T}}$ where $T\in\Ts_\lambda$ and $\F_{x,T}=t^{-1}(T)$.
\end{lem}
For a proof, see \cite{Sp1} or \cite{H}. Spaltenstein discusses this very briefly, and uses a slightly different convention, as was also mentioned earlier. In his result the indexing set is a subset of the standard tableaux. It can be seen that this subset consists of what we will define later in this paper to be the \emph{standardization} of the semistandard tableaux. 

On the other hand Haines, during the proof of Theorem 3.1 proves a more general result about irreducible components of fibers of convolution morphisms from convolution product of $G(\mathcal{O})$-orbits in the affine Grassmannian. In his result, the combinatorial data are sequences of dominant weights such that the difference of two consecutive weights is in the orbit of the Weyl group acting on a dominant minuscule weight. In our case these correspond to the semistandard tableaux.

\subsection{Relative Position, Words and Arrays}\label{relposwa} 
Given two flags $F$, $F'$, we have defined in \eqref{rpos} their relative position $M(F,F')$. Notice that if $F\in\F^\mu$ and $F'\in\F^\nu$, the row sums of this matrix will be $\mu=(\mu_1,\ldots,\mu_n)$ and the column sums will be $\nu=(\nu_1,\ldots,\nu_m)$. Then, see \cite[1.1]{blm}, the set $\MM{\mu}{\nu}$ of all such matrices parametrizes the orbits of the  diagonal action of $\gl_d$ on $\F^\mu\times\F^\nu$.

In particular, if $F$ and $F'$ are both complete flags in $V$, $M(F,F')$ will be a permutation matrix. This data is equivalent to the word $w(F,F')=w(1)\ldots w(d)$  where $w(i)=j$ if $1$ appears in the $(j,i)$-entry of the matrix. 


\begin{defn}\label{arrtom}If $F,F'$ are both partial flags, then $M(F,F')$ is just a matrix of nonnegative integers. We can record the same data in a two-rowed array 
$$\omega=\left(\begin{array}{cccc}u(1) & u(2) & \ldots & u(d) \\ w(1) & w(2) & \ldots & w(d) \end{array}\right)$$
which is defined as follows. 

A pair $\left(\begin{smallmatrix} i \\ j\end{smallmatrix}\right)$ appears in $\omega$ a number of times equal to the $(j,i)$-entry of $M(F,F')$. 

The array $\omega$ is then ordered so that it satisfies the following relation: 
\begin{equation}\label{order}u(1)\leq u(2)\leq\ldots\leq u(d)\qquad\text{ and }\quad w(k)\geq w(k+1)\quad\text{ if }u(k)=u(k+1).\end{equation}
\end{defn}
\begin{exa}\label{notlex}If $M(F,F')$ is the matrix on the left, the corresponding array $\omega$ is given on the right:
$$M(F,F')=\left(\begin{array}{ccc} 1 & 0 & 2 \\ 3 & 1 & 1 \end{array}\right)\quad \omega=\left(\begin{array}{cccccccc} 1 & 1 & 1 & 1 & 2 & 3 & 3 & 3 \\ 2 & 2 & 2 & 1 & 2 & 2& 1 & 1 \end{array}\right).$$ \end{exa}
The set $\MM{\mu}{\nu}$ 
is thus identified, with the convention just described, 
with the set of two rowed arrays such that the first row has content $\nu$, the second row has content $\mu$, and they satisfy the order \eqref{order}. 

Depending on what is more convenient at each time, we will use either description of this set.
\begin{rem}\label{dcosets}Another way of looking at the set $\MM{\mu}{\nu}$ is as the set of double cosets $S_{\mu}\setminus S_d/S_{\nu}$. Here $S_d$ is the symmetric group on $d$ letters and $S_{\mu}$ and $S_{\nu}$ are the Young subgroups corresponding to the compositions $\mu$ and $\nu$.
\end{rem}
\begin{rem}\label{lex}Our convention is different from what is used in \cite{F} and \cite{S2}, where the arrays are taken to be in \emph{lexicographic order},
that is with
$$u(1)\leq u(2)\leq\ldots\leq u(d)\qquad\text{ and }\quad w(k)\leq w(k+1)\quad\text{ if }u(k)=u(k+1).$$
With the lexicographic convention, the matrix of example \ref{notlex} would correspond to the array
$$\omega'=\left(\begin{array}{cccccccc} 1 & 1 & 1 & 1 & 2 & 3 & 3 & 3\\ 1 & 2 & 2 & 2 & 2 & 1 & 1 & 2\end{array}\right). $$ 
\end{rem}
\section{Robinson-Schensted-Knuth Correspondence and Standardization} In this section, we will review quickly some definitions and properties of the RSK correspondence, following mainly the conventions of \cite[I]{F} and \cite[7.11]{S2}. Then we will see how to adapt the results to the conventions we are using.
\subsection{Review of RSK}
Just for this review, we will call a tableau \emph{semistandard} if it is weakly increasing along rows and strictly increasing down columns. With this convention, the tableaux we defined in section \ref{pfsst} are transposed of semistandard tableaux.
We will also identify matrices with arrays using the lexicographic order, as in Remark \ref{lex}. 

With increasing generality, the RSK correspondence gives a bijection between permutations and pairs of standard tableaux of same shape, or between two-rowed arrays in lexicographic order and pairs of semistandard tableaux of same shape.

Given a permutation word $w$ or a two rowed array $\omega$, where
$$w=w(1)\ldots w(d)\quad\quad\omega=\left(\begin{smallmatrix}u(1) & u(2) & \ldots & u(d) \\ w(1) & w(2) & \ldots & w(d) \end{smallmatrix}\right),$$ the algorithm is given by inserting the entries of the word (or of the second row of the array) by row bumping in the first tableau. 
At the same time we record in the second tableau which box has been added at each step (in the more general case of the array, the added box at the $k$-th step will be recorded with $u(k)$ as opposed to $k$). The convention for row bumping is that a new entry $z$ bumps the left-most entry in the row which is \emph{strictly larger} than $z$.

If $T$, $T'$ are semistandard tableaux and $\omega$ is an array in lexicographic order, we will denote the correspondence by
$$ M(T,T')=\omega; \quad\text{ or }\quad(T,T')\stackrel{RSK}{\longleftrightarrow}\omega.$$

As can be seen in \cite[7.11]{S2}, given a semistandard tableau $T$ we can consider its \emph{standardization} $\tilde{T}$. It is a standard tableau of the same shape as $T$. We construct it in this way: the $\mu_1$ boxes that contain $1$ in $T$ will be replaced by the numbers $1,2,\ldots,\mu_1$ increasingly from left to right. Then the boxes that originally contained $2$'s will be replaced by $\mu_1+1,\ldots,\mu_1+\mu_2$ also increasingly from left to right, and so on.
\begin{exa}\label{ex1}
$$T=\young(112,2,3)\qquad\tilde{T}=\young(124,3,5).$$
\end{exa}
In a similar way, given an array in lexicographic order $\omega=\left(\begin{smallmatrix}u(1) & u(2) & \ldots & u(d) \\ w(1) & w(2) & \ldots & w(d) \end{smallmatrix}\right)$ we can define the standardization $\tilde{\omega}$. It is given by replacing $u(i)$ with $i$ in the first row, while in the second row we replace the $1$'s with $1,2,\ldots,\mu_1$ increasing from left to right, then the $2$'s and so on. The standardization of an array will then be a permutation.
\begin{exa}\label{ex2}
$$\omega=\left(\begin{array}{ccccc} 1 & 2 & 2 & 3 & 3 \\ 3 & 1 & 2 & 1 & 2 \end{array}\right)\qquad \tilde{\omega}=\left(\begin{array}{ccccc} 1 & 2 & 3 & 4 & 5 \\ 5 & 1 & 3 & 2& 4\end{array}\right).$$
\end{exa}
Standardization allows us to always reduce the RSK correspondence to the special case of permutations and standard tableaux, because standardization and RSK commute.
\begin{lem}\label{lem1}The following diagram commutes:
$$\begin{CD}
\Ts_\lambda\times\T_\lambda^{\nu} @>\text{RSK}>> M^{\mu,\nu}(\mathbb{Z}_{\geq 0}) \\
@VV\std\times\std V                             @VV\std V \\
\T_\lambda\times\T_\lambda @>\text{RSK}>>S_d
\end{CD}$$
\end{lem}
In the diagram, $\T_\lambda$, $\Ts_\lambda$, $\T_\lambda^{\nu}$ are respectively the set of standard tableaux and the sets of semistandard tableaux with content $\mu$ and $\nu$, all of shape $\lambda$. Also, $M^{\mu,\nu}(\mathbb{Z}_{\geq 0})$ is the set of two rowed arrays in lexicographic order with row contents $\nu$ and $\mu$ and $\std$ is the standardization map.

The lemma is proved in \cite[7.11.6]{S2}, but let us illustrate this with an example. 
\begin{exa}Let $T$, $\omega$ be as in examples \ref{ex1} and \ref{ex2} and let 
$$T'=\young(123,2,3)\quad\text{ then we have }\quad \std(T')=\tilde{T'}=\young(135,2,4)$$
then $(T,T')\stackrel{RSK}{\longleftrightarrow}\omega$
and indeed $(\tilde{T},\tilde{T'})\stackrel{RSK}{\longleftrightarrow}\tilde{\omega}$.
\end{exa}
\subsection{Variation on RSK}\label{varrsk} In this paper we will need a slight variation on the RSK correspondence. This will agree with RSK on permutations, but will give different results in the case of general two rowed arrays. It will associate to an array satisfying \eqref{order}, a pair of tableaux that are strictly increasing along rows and weakly increasing down columns. This is what we called \emph{semistandard} in section \ref{pfsst} and we will keep using this terminology from now on. In the rest of this paper, we will also set the convention of identifying matrices and arrays using Definition \ref{arrtom}.

The variation of the correspondence is defined modifying the row bumping algorithm to the following: a new entry $z$ will bump the left-most entry in the row which is \emph{greater or equal} to $z$. The recording tableau will be constructed in the usual way. 

This difference is clearly irrelevant in the case of standard tableaux, but our new choice of row bumping will produce tableaux that are strictly increasing along rows and weakly increasing down columns. This is similar to the dual RSK defined in \cite[7.14]{S2}, which however is only defined for matrices of $0$'s and $1$'s.

Since we will only use this variation on the correspondence, from now on we will call this one RSK and we will use the same notation as before, there should be no confusion.

\begin{lem}This procedure gives a bijection between matrices of non-negative integers and pairs of semistandard (strictly increasing along rows and weakly increasing down columns) tableaux of same shape.\end{lem} 
\begin{proof}This is completely analogous to the usual proofs of the RSK correspondence (see \cite{F},\cite{S2}).

If the array corresponding to the matrix is $\omega=\left(\begin{array}{ccc}u(1) & \ldots & u(d) \\ w(1) & \ldots & w(d) \end{array}\right)$ and by the correspondence it gives us the pair of tableaux $(P,Q)$, then it is clear that the insertion tableau $P$ will be semistandard. To check that the recording tableau $Q$ is also semistandard, it is enough to show that if $u(i)=u(i+1)$, then $u(i+1)$ will end up in a row of $Q$ that is strictly below the row of $u(i)$. 

Since $\omega$ satisfies \eqref{order}, if $u(i)=u(i+1)$, then $w(i)\geq w(i+1)$. This means that if $w(i)$ bumps an element $y_i$ from the first row, then the element $y_{i+1}$ bumped by $w(i+1)$ from the first row must be in the same box where $y_i$ was or in a box to the left of it. In turn, this implies that $y_i\geq y_{i+1}$ and we can iterate this argument for the following rows. Now, the bumping route $R_{i}$ of $w(i)$ must stop before the bumping route $R_{i+1}$ of $w(i+1)$, which will then continue at least one row below that of $R_i$, which shows what we want.  

The fact that the correspondence is a bijection just follows from the fact that we can do the reverse row bumping algorithm by taking at each step the box that in the recording tableau contains the biggest number. In case of equal elements, we will take the one that is in the lowest row.
\end{proof}
\begin{rem}
Basically in this version of RSK we are considering equal entries in a tableau to be 'bigger' if they are in a lower row and, while inserting, sequences of equal numbers are considered decreasing sequences.\end{rem}

This leads us to a new definition of \emph{standardization} that will give us an analogous result to lemma \ref{lem1}. 
Given a semistandard tableau $T$, we define its standardization $\tilde{T}$ by replacing the $1$'s with $1,2,\ldots,\mu_1$ \emph{starting from the top row and going down}, and then the same for $2$'s and so on. For an array $\omega$ ordered as in \eqref{order}, we define $\tilde{\omega}$ by replacing the first row with $1,2,\ldots,d$ and on the second row we replace the $1$'s by $1,2,\ldots,\mu_1$ \emph{decreasingly} from left to right and same for the rest, always decreasing from left to right.
\begin{exa}\label{exsst}
$$T=\young(12,12,3)\qquad \tilde{T}=\young(13,24,5)$$
$$\omega=\left(\begin{array}{ccccc} 1 & 2 & 2 & 3 & 3 \\ 1 & 3 & 1 & 2 & 2 \end{array}\right)\qquad
\tilde{\omega}=\left(\begin{array}{ccccc} 1 & 2 & 3 & 4 & 5 \\ 2 & 5 & 1 & 4 & 3 \end{array}\right) $$
\end{exa}
\begin{rem}From the point of view of Remark \ref{dcosets}, the standardization of an array corresponds to choosing the longest representative for the double coset.\end{rem}
With our new conventions for semistandard tableaux, order of arrays, RSK, standardization and the same notation of lemma \ref{lem1} we have that
\begin{lem}\label{lem2}Standardization and RSK commute, as in the following diagram:
$$\begin{CD}
\Ts_\lambda\times\T^{\nu}_\lambda @>\text{RSK}>> M^{\mu,\nu}(\mathbb{Z}_{\geq 0}) \\
@VV\std\times\std V                             @VV\std V \\
\T_\lambda\times\T_\lambda @>\text{RSK}>>S_d
\end{CD}$$
\end{lem}
The proof, mutatis mutandis, is the same as the proof of lemma \ref{lem1} in \cite[7.11.6]{S2}. It is just the observation that the standardization we choose for the arrays is exactly the one that makes the insertion procedure work the way we want, turning sequences of equal numbers into decreasing sequences.
\begin{exa}Let $T$, $\omega$ as in example \ref{exsst} and let
$$T'=\young(12,23,3)\quad\text{ then we have }\quad \tilde{T'}=\young(12,34,5)$$
then $(T,T')\stackrel{RSK}{\longleftrightarrow}\omega$ and $(\tilde{T},\tilde{T'})\stackrel{RSK}{\longleftrightarrow}\tilde{\omega}$.
\end{exa}
\begin{rem}\label{rmk}It is clear that if we fix the contents $\mu$ and $\nu$, two different arrays $\omega_1\neq\omega_2\in M^{\mu,\nu}(\mathbb{Z}_\geq 0)$ when standardized will give two different permutations $\tilde{\omega}_1\neq\tilde{\omega}_2$. That is we have an injective map
$$\std:\MM{\mu}{\nu}\to S_d.$$
We therefore have an inverse
$$\std^{-1}:\std(\MM{\mu}{\nu})\to \MM{\mu}{\nu}$$
which is easily described as follows:
$$\left(\begin{array}{cccccc} 1 & 2 & \ldots & \nu_1 & \nu_1 +1 & \ldots \\ w(1) & w(2) & \ldots & w(\nu_1) & w(\nu_1+1) & \ldots \end{array}\right)
\mapsto \left(\begin{array}{cccccc} 1 & 1 & \ldots & 1 & 2 & \ldots \\ w'(1) & w'(2) & \ldots & w'(\nu_1) & w'(\nu_1+1) & \ldots\end{array}\right)
 $$
the first row is just replaced by $\nu_1$ $1$'s, followed by $\nu_2$ $2$'s and so on, while we have 
$$w'(k)=j\quad\text{ if }\quad w(k)\in\{\mu_1+\ldots+\mu_{j-1}+1,\ldots,\mu_1+\ldots+\mu_j\}.$$ 
\end{rem}
\section{RSK and Partial Flag Varieties} In this section we will use all the conventions of section \ref{varrsk} and the notations of section \ref{pfsst}.

We state and prove the main result of the first part of the paper, which generalizes Theorem \ref{stein}. The strategy for the proof is to use standardization and Lemma \ref{lem2} to reduce the problem to the case of complete flags.
\begin{thm}\label{rosso}Let $x\in\End(V)$ be a nilpotent transformation of Jordan type $\lambda$, $T\in\Ts_\lambda$, $S\in\T^{\nu}_\lambda$ be semistandard tableaux, and let $C_{x,T}$ and $C_{x,T'}$ be respectively the irreducible components of $\F^\mu_x$ and $\F^\nu_x$ corresponding to the tableaux $T$ and $T'$. 

Then, for generic $F\in C_{x,T}$ and $F'\in C_{x,T'}$, we have that the relative position matrix $M(F,F')$ is the same as the matrix $M(T,T')$ given by the RSK correspondence. \end{thm}
\begin{proof}
For a fixed $\mu=(\mu_1,\ldots,\mu_n)$ with $|\mu|=\mu_1+\ldots+\mu_n=d$, consider the map 
$$p_\mu:\F\to\F^\mu$$ 
that forgets some of the spaces, that is
$$(0=F_0,F_1,F_2,\ldots,F_{n-1},F_n=V)\mapsto (0=F_0,F_{\mu_1},F_{\mu_1+\mu_2},\ldots,F_{\mu_1+\ldots+\mu_{n-1}},F_n=V).$$

Clearly, if $F$ is any partial flag in $\F^\mu_x$ and $\tilde{F}\in p_\mu^{-1}(F)$, then $\tilde{F}\in\F_x$ because for all $j$ there is some $i$ such that
$$\tilde{F}_{\mu_1+\ldots+\mu_i}\subset\tilde{F}_{j-1}\subset\tilde{F}_j\subset\tilde{F}_{\mu_1+\ldots+\mu_{i+1}}$$
and 
$$x(\tilde{F}_j)\subset x(\tilde{F}_{\mu_1+\ldots+\mu_{i+1}})=x(F_{i+1})\subset F_i=\tilde{F}_{\mu_1+\ldots+\mu_i}\subset\tilde{F}_{j-1}.$$

Now, let $t:\F_x^\mu\to\Ts_\lambda$ be the map that associates a semistandard tableau to a partial flag, as in Definition \ref{eqdeft}. 

We fix a semistandard tableau $T$ and we let $\F_{x,T}:=t^{-1}(T)$, then $\F_{x,T}$ is a constructible dense subset of $C_{x,T}$.

Let $\tilde{T}$ be the standardization of $T$ and let $\F_{x,\tilde{T}} =t^{-1}(\tilde{T})\subset\F_x$ be the dense subset of $C_{x,\tilde{T}}$. The set $C_{x,\tilde{T}}$ is the irreducible component of the complete flag variety associated to the standard tableau $\tilde{T}$.

It is clear that if $\tilde{F}\in \F_{x,\tilde{T}}$, then we have $F=p_\mu(\tilde{F})\in \F_{x,T}$ because
$$x|_{F_i}=x|_{\tilde{F}_{\mu_1+\ldots+\mu_i}}$$
also, the map
$$p_\mu:\F_{x,\tilde{T}}\to \F_{x,T}$$
is surjective. This is because we can always find appropriate subspaces to complete a partial flag $F$ to a flag $\tilde{F}$ such that the restriction of $x$ to those subspaces has the Jordan type we want.

What we have said so far applies in the same way if we fix another semistandard tableau $T'$ of content $\nu$ and we consider the sets $\F_{x,T'}\subset\F^\nu_x$ and $\F_{x,\tilde{T'}}\subset\F_x$.

Now, let us fix two semistandard tableaux $T$ and $T'$ as in the statement of the theorem, and consider their standardizations $\tilde{T}$ and $\tilde{T'}$. For general complete flags $\tilde{F}\in C_{x,\tilde{T}}$ and $\tilde{F'}\in C_{x,\tilde{T'}}$, Theorem \ref{stein} tells us that $M(\tilde{F},\tilde{F'})=M(\tilde{T},\tilde{T'})$. We let then $X_{\tilde{T}}\subset C_{x,\tilde{T}}$, $X_{\tilde{T'}}\subset C_{x,\tilde{T'}}$ be the open dense subsets such that this is true.

Then $X_{\tilde{T}}\cap \F_{x,\tilde{T}}$ is constructible dense in $C_{x,\tilde{T}}$. Hence it contains an open dense subset and the image of
$$p_\mu:X_{\tilde{T}}\cap \F_{x,\tilde{T}}\to \F_{x,T}$$
is constructible dense in $\F_{x,T}$, therefore it is also dense in $C_{x,T}$. In the same way, $p_\nu(X_{\tilde{T'}}\cap \F_{x,\tilde{T'}})$ is constructible dense in $\F_{x,T'}$.
\begin{claim}If $F\in p_\mu(X_{\tilde{T}}\cap \F_{x,\tilde{T}})$ and $F'\in p_\nu(X_{\tilde{T'}}\cap \F_{x,\tilde{T'}})$ then $M(F,F')=M(T,T')$. \end{claim}
Let $\tilde{F}\in p_\mu^{-1}(F)$ and $\tilde{F'}\in p_\nu^{-1}(F')$, then by Lemma \ref{lem2} we have that 
$$\tilde{\omega}=\std(M(T,T'))=M(\tilde{T},\tilde{T'}).$$

Now let $\omega'=M(F,F')$. By the definition of relative position of flags, the array $\tilde{\omega}=M(\tilde{T},\tilde{T'})=M(\tilde{F},\tilde{F'})$ is such that for all $i,j$
$$ \card \left\{\begin{tabular}{r|l}
\multirow{2}{*}{$\left(\begin{array}{c} \tilde{u} \\ \tilde{w} \end{array}\right)\in \tilde{\omega}$}
& $\tilde{u}\in\{\nu_1+\ldots+\nu_{j-1}+1,\ldots,\nu_1+\ldots\nu_j\}$, \\
& $\tilde{w}\in\{\mu_1+\ldots+\mu_{i-1}+1,\ldots,\mu_1+\ldots+\mu_i\}$
\end{tabular}\right\} $$

\begin{align*} &= \dim\left(\df{\tilde{F}_{\mu_1+\ldots+\mu_i}\cap\tilde{F'}_{\nu_1+\ldots+\nu_j}}{(\tilde{F}_{\mu_1+\ldots+\mu_{i-1}}\cap \tilde{F'}_{\nu_1+\ldots+\nu_j}) + (\tilde{F}_{\mu_1+\ldots+\mu_{i}}\cap \tilde{F'}_{\nu_1+\ldots+\nu_{j-1}})}\right) \\
 &= \dim\left(\df{F_i\cap F'_j}{F_i\cap F'_{j-1}+F_{i-1}\cap F'_j}\right) \\
 &= \card\left\{\left(\begin{array}{c} u \\ w \end{array}\right)\in\omega' \left| \left(\begin{array}{c} u \\ w \end{array}\right)\right.= \left(\begin{array}{c} i \\ j \end{array}\right)\right\}
 \end{align*}
Therefore, by Remark \ref{rmk}, $\omega'=\std^{-1}(\tilde{\omega})$. It follows that $\std(\omega')=\tilde{\omega}$, that is 
$$\std(M(F,F'))=\std(M(T,T')).$$ 
Again by Remark \ref{rmk}, this implies that $M(F,F')=M(T,T')$. This concludes the proof of the claim. 

Since $p_\mu(X_{\tilde{T}}\cap \F_{x,\tilde{T}})$ and $p_\nu(X_{\tilde{T'}}\cap \F_{x,\tilde{T'}})$ are constructible dense in $C_{x,T}$ and $C_{x,T'}$ respectively, they each contain an open dense subset of the respective irreducible component, which proves the theorem.
\end{proof}

\section{Mirabolic Flag Varieties}
With this section, we start the second part of this paper, where we generalize the construction of Travkin (see \cite{T}). We keep the notation of Section \ref{notat}.

\subsection{$\gl(V)$-orbits in $\F^{\mu}\times\F^{\mu'}\times V$}\label{sec2}

Let $\mu=(\mu_1,\ldots,\mu_n)$, $\mu'=(\mu'_1,\ldots,\mu'_{n'})$ be two compositions of $d$. We consider the diagonal $G$-action on the set $\F^\mu\times\F^{\mu'}\times V$. So, let $(F,F',v)\in\F^{\mu}\times\F^{\mu'}\times V$ and look at the orbit $G\cdot(F,F',v)$. 

If $v=0$, this orbit lies in $\F^{\mu}\times\F^{\mu'}\times\{0\}\simeq\F^\mu\times\F^{\mu'}$. As in Section \ref{relposwa} we parametrize such orbits  by the set $\MM{\mu}{\mu'}$ of matrices with row sums $\mu$ and column sums $\mu'$ which we can also identify with the set of two rowed arrays of positive integers with row contents $\mu'$ and $\mu$.

If $v\neq 0$, the orbit $G\cdot(F,F',v)$ is the preimage of an orbit in $\F^{\mu}\times\F^{\mu'}\times\mathbb{P}(V)$. This is because for all $c\in k^{\times}$, $(F,F',cv)=c\Id\cdot(F,F',v)\in G\cdot(F,F',v)$.

The $G$-orbits on $\F^{\mu}\times\F^{\mu'}\times\mathbb{P}(V)$ have been parametrized in \cite[2.11]{MWZ} (see also \cite[2.2]{M}) by ``decorated matrices''. These are pairs $(M,\Delta)$, where $M$ is a matrix in $\MM{\mu}{\mu'}$ and $\Delta=\{(i_1,j_1),\ldots,(i_k,j_k)\}$ is a nonempty set  that satisfies 
$$1\leq i_1<\ldots <i_k\leq n,\qquad 1\leq j_k<\ldots <j_1\leq n'$$ 
and such that the entry $M_{ij}>0$ for all $(i,j)\in\Delta$.

We can concisely write down a pair $(M,\Delta)$, in a similar way to what is done in \cite{M}, by parenthesizing the entries of the matrix corresponding to $\Delta$.
\begin{exa}\label{exM1}$$M=\begin{pmatrix} 1 & 0 & 2 \\ 1 & 1 & 0 \\ 0 & 3 & 0 \end{pmatrix};\qquad\Delta=\{(1,3),(2,1)\}$$
$$(M,\Delta)=\begin{pmatrix} 1 & 0 & (2) \\ (1) & 1 & 0 \\ 0 & 3 & 0 \end{pmatrix} $$
\end{exa}
\begin{lem}\label{decmat}There is a 1-1 correspondence between the set of pairs $(M,\Delta)$ as above and the set of pairs $(\omega,\beta)$ where $\omega$ is a two rowed array and $\beta\subset\{1,\ldots,d\}$ is a nonempty subset such that if $i\in\{1,\ldots,d\}\setminus\beta$ and $j\in\beta$, either $u(i)>u(j)$ or $w(i)>w(j)$.\end{lem}
\begin{proof}The correspondence between $M$ and $\omega$ is just the identification we discussed in Definition \ref{arrtom}. Now, consider the map
$$\varphi:\{1,\ldots,d\}\to\{1,\ldots,n\}\times\{1,\ldots,n'\}$$
$$l\mapsto (w(l),u(l)).$$
Then $\Delta$ will be the subset of $\varphi(\beta)$ defined by
$$\Delta=\{(i,j)\in\varphi(\beta)|(i+1,j)\notin\varphi(\beta) \text{ or }(i,j+1)\notin\varphi(\beta)\}.$$
Given $\Delta$ we can recover $\beta$ in the following way: let
$$\Delta'=\{(i,j)\in\{1,\ldots,n\}\times\{1,\ldots,n'\}|\exists (i_0,j_0)\in\Delta \text{ s.t. }i\leq i_0\text{ or }j\leq j_0\}$$
then $\beta=\varphi^{-1}(\Delta')$.

It is not difficult to see that these definitions give inverse correspondences. 

Visually, $\varphi(\beta)$ identifies a set of positions in the matrix that fits in a Young diagram, and such that no other nonzero positions are in the diagram. The set $\Delta$ consists then of the outer corners of that diagram. 

Vice versa, given $\Delta$, $\Delta'$ is the set of all positions of the matrix weakly northwest of $\Delta$. Then $\beta=\varphi^{-1}(\Delta')$ consists of all the columns of the array corresponding to the nonzero positions in $\Delta'$.
\end{proof}
\begin{exa}If we take the decorated matrix $(M,\Delta)$ of Example \ref{exM1}, we have that
$$\Delta'=\{(1,1),(1,2),(1,3),(2,1)\};\quad\omega=\begin{pmatrix} 1 & 1 & 2 & 2 & 2 & 2 & 3 & 3 \\ 2 & 1 & 3 & 3 & 3 & 2 & 1 & 1 \end{pmatrix}$$
Then $\beta=\{1,2,7,8\}$.
\end{exa}
\begin{defn}\label{decarr}We define the set $\D^{\mu,\mu'}$ of ``decorated arrays'' to be the set of all pairs $(\omega,\beta)$, where $\omega\in\MM{\mu}{\mu'}$ and $\beta\subset\{1,\ldots,d\}$ is a (possibly empty) subset such that if $i\in\{1,\ldots,d\}\setminus\beta$ and $j\in\beta$, then either $u(i)>u(j)$ or $w(i)>w(j)$.\end{defn}
By Lemma \ref{decmat}, the set of decorated matrices (if we also allow $\Delta=\emptyset$) and decorated arrays are identified, so we might use either description of the set, depending on what is most convenient at each time. 

By the result in \cite[2.11]{MWZ} and Lemma \ref{decmat}, we can then parametrize the $G$-orbits on $\F^{\mu}\times\F^{\mu'}\times V$ with the set $\D^{\mu,\mu'}$. The pairs $(\omega,\beta)\in\D^{\mu,\mu'}$ with $\beta\neq\emptyset$ correspond to the $G$-orbits in $\F^{\mu}\times\F^{\mu'}\times\mathbb{P}(V)$, and the ones with $\beta=\emptyset$ correspond to the case of $v=0$. 

We are going to give a direct proof of this parametrization. In order to do that, we will use the following result of Travkin (\cite[Lemma 1]{T}).
\begin{lem}\label{lemM1} Let $A\subset\End(V)$ be an associative algebra with identity and $A^{\times}$ the multiplicative group of $A$. Suppose that the $A$-module $V$ has finitely many submodules. Then the $A^{\times}$-orbits in $V$ are in 1-1 correspondence with these submodules. Namely, each $A^{\times}$-orbit has the form
$$\Omega_S:=S\setminus \bigcup_{S'\subsetneq S}S'$$
where $S$ is an $A^{\times}$-submodule of $V$ and the union is taken over all proper submodules of $S$.
\end{lem}
\begin{prop}There is a 1-1 correspondence between $G$-orbits in $\F^{\mu}\times\F^{\mu'}\times V$ and the set $\D^{\mu,\mu'}$.
\end{prop}
\begin{proof}For each $\omega=\left(\begin{array}{cccc}u(1) & u(2) & \ldots & u(d) \\ w(1) & w(2) & \ldots & w(d) \end{array}\right)$ in $\MM{\mu}{\mu'}$, let $\Omega_\omega$ be the corresponding $G$-orbit in $\F^{\mu}\times\F^{\mu'}$. 

In particular, $(F,F')\in\Omega_\omega$ if and only if there exists a basis $\{e_i|i=1,\ldots,d\}$ of $V$ such that
\begin{align}\label{base}F_i= & \seq{e_r|w(r)\leq i} \\
\notag F'_j= & \seq{e_s|u(s)\leq j}.
\end{align}
For a fixed $\omega$, consider a point $(F,F')\in\Omega_\omega$ and let $H$ be its stabilizer in $G$. Then the $H$-orbits in $V$ are in 1-1 correspondence with the $G$-orbits of $\F^{\mu}\times\F^{\mu'}\times V$ consisting of points $(D,D',v)$ with $(D,D')\in\Omega_\omega$.

Let $A_F, A_{F'}\subset\End(V)$ respectively be the subalgebras that leave the partial flags $F, F'$ invariant, i.e.
$$A_F:=\{a\in\End(V)|a(F_i)\subset (F_i)\quad\forall i\}$$
$$A_{F'}:=\{a\in\End(V)|a(F'_i)\subset (F'_i)\quad\forall i\}$$
Let $A=A_F\cap A_{F'}$, then $H=A^{\times}$.
So pick a basis $\{e_i\}$ of $V$ satisfying \eqref{base}, and let $E_{ij}$ be the linear operator such that
$$E_{ij}e_r=\delta_{jr}e_i.$$
Then 
$$A=\bigoplus_{u(i)\leq u(j);w(i)\leq w(j)}k E_{ij}.$$
From this it follows that all the $A$-submodules of $V$ have the form $S(\beta):=\oplus_{i\in\beta}ke_i$, where $\beta$ is like in Definition \ref{decarr}. In particular, they are finite, so we can apply lemma \ref{lemM1} to conclude the proof.
\end{proof}
\begin{defn}We will denote by $\Omega_{\omega,\beta}$ the $G$-orbit in $\F^{\mu}\times\F^{\mu'}\times V$ corresponding to $(\omega,\beta)$.
\end{defn}
\begin{rem}The orbit $\Omega_{\omega,\beta}$ consists exactly of the triples $(F,F',v)$ such that there exists a basis $\{e_i|i=1,\ldots,d\}$ of $V$ that satisfies \eqref{base} and with
$$v=\sum_{i\in\beta}e_i.$$
\end{rem} 
\subsection{Conormal Bundles and Mirabolic RSK}
We consider the variety $\Xmm:=\F^{\mu}\times\F^{\mu'}\times V$, and its cotangent bundle $T^*(\Xmm)$. We know that, see \cite[4.1.2]{CG},
$$T^*(\F^{\mu})=\{(F,x)\in\F^{\mu}\times\N | x(F_i)\subset F_{i-1}\quad\forall i\}.$$
Therefore
$$T^*(\Xmm)=\{(F,F',v,x,x',v^*)\in\Xmm\times\N\times\N\times V^*|F\in\F^\mu_x;\quad F'\in\F^{\mu'}_{x'}\}.$$
We have the moment map
\begin{align*}T^*(\Xmm)& \to\gL(V)^*\simeq\gL(V) \\
(F,F',v,x,x',v^*) & \mapsto x+x'+v^*\otimes v.\end{align*}
We let $\Ymm$ be the preimage of $0$ under the moment map, then $\Ymm$ is the union of the conormal bundles of the $G$-orbits in $\Xmm$: 
$$\Ymm:=\{(F,F',v,x,x',v^*)\in T^*(\Xmm)|x+x'+v^*\otimes v=0\}=\bigsqcup_{\omega,\beta}N^*\Omega_{\omega,\beta}.$$
Hence, all the irreducible components of $\Ymm$ are the closures $\overline{N^*\Omega_{\omega,\beta}}$.

Now, consider the variety $Z$ of quadruples
$$Z:=\{(x,x',v,v^*)\in \N\times\N\times V\times V^*|x+x'+v^*\otimes v=0\}.$$
We then have a projection
\begin{align*}\pi:\Ymm & \to Z \\
(F,F',v,x,x',v^*) & \mapsto (x,x',v,v^*).\end{align*}
We let $\PP$ be the set of pairs of partitions $(\lambda,\theta)$ such that $|\lambda|=d$ and $\lambda_i\geq\theta_i\geq\lambda_{i+1}$ for all $i$.
\begin{rem}The set $\PP$ parametrizes $G$-orbits on $\N\times V$, as is proved independently in both \cite[Theorem 1]{T} and \cite[Proposition 2.3]{AH}. In particular, $(x,v)$ is in the orbit corresponding to $(\lambda,\theta)$ if the Jordan type of $x$ is $\lambda$ and the Jordan type of $x|_{V/k[x]v}$ is $\theta$.\end{rem} 
Define the set of triples $\TT:=\{(\lambda,\theta,\lambda')|(\lambda,\theta)\in\PP; (\lambda',\theta)\in\PP\}$. 

For any $\textbf{t}=(\lambda,\theta,\lambda')\in\TT$, we write $Z^{\textbf{t}}$ for the subset of quadruples $(x,x',v,v^*)\in Z$ such that the Jordan types of $x,x'$ and $x|_{V/k[x]v}$ are respectively $\lambda, \lambda'$ and $\theta$.
\begin{rem}Notice that in the previous statement we did not break any symmetry by choosing $x$ instead of $x'$, because if $x+x'+v^*\otimes v=0$, then $k[x]v=k[x']v$ and $x|_{V/k[x]v}=-x'|_{V/k[x']v}$.
\end{rem}

Now if $\tilde{\omega}=(\omega,\beta)\in\D^{\mu,\mu'}$,  we can consider a point $y=(F,F',v,x,x',v^*)$ in the variety $Y_{\tilde{\omega}}:=\overline{N^*(\Omega_{\tilde{\omega}})}$. In particular we can take $y\in N^*(\Omega_{\tilde{\omega}})$.

Then $\pi(y)\in \Zt$ for some $\textbf{t}=(\lambda, \theta, \lambda')\in\TT$. Now, $\Zt$ is irreducible, as is shown in the proof of Proposition 1 in \cite{T}. Hence this $\textbf{t}=\textbf{t}(y)$ will be the same for all $y$ in an open dense subset of $Y_{\tilde{\omega}}$. With this choice of $y$, we can then denote $\textbf{t}=\textbf{t}(\tlo)$ to emphasize that it depends only on $\tilde{\omega}$. Let $T=T(\tilde{\omega})\in\T_\lambda^{\mu}$ and $T'=T'(\tilde{\omega})\in\T_{\lambda'}^{\mu'}$ such that $F=F(y)\in\F_{x,T}$ and $F'=F'(y)\in\F_{x',T'}$.

\begin{prop}\label{prop1}The assignment $\tlo\mapsto (\emph{\textbf{t}}(\tlo),T(\tlo),T'(\tlo))$, that we just described, gives a 1-1 correspondence 
$$\D^{\mu,\mu'}\longleftrightarrow \{((\lambda,\theta,\lambda'),T,T')|(\lambda,\theta,\lambda')\in\emph{\TT}, T\in\T_\lambda^\mu, T'\in\T_{\lambda'}^{\mu'}\}.$$ \end{prop}
\begin{proof}Consider the set $Y^{\textbf{t},T,T'}\subset\Ymm$ defined by
$$Y^{\textbf{t},T,T'}=\{y\in\Ymm|\pi(y)\in\Zt, F(y)\in\F_{x,T}, F'(y)\in\F_{x',T'}\}.$$ 
Then for all $(\textbf{t},T,T')$, $Y^{\textbf{t},T,T'}$ is a locally closed subset of $\Ymm$ and
$$\Ymm=\bigsqcup_{\textbf{t},T,T'}Y^{\textbf{t},T,T'}.$$
\begin{claim} These locally closed subsets are irreducible and $\dim Y^{\textbf{t},T,T'}=\dim\Ymm$ for all $\textbf{t},T,T'$.\end{claim}
We look at the projection $\pi|_{Y^{\textbf{t},T,T'}}:Y^{\textbf{t},T,T'}\to\Zt$. All the fibers of this map are of the form 
$$\pi^{-1}((\textbf{t},T,T'))=\{y\in Y^{\textbf{t},T,T'}|F(y)\in\F_{x,T},F'(y)\in\F_{x',T'}\}\simeq\F_{x,T}\times\F_{x',T'}.$$
It follows that they are irreducible and they have the same dimension. The set $\Zt$ is also irreducible, hence the sets $Y^{\textbf{t},T,T'}$ are irreducible. 

From now on in the paper we will use the notation $(a^{d_1},b^{d_2},\ldots)$ for the sequence $(a,\ldots,a,b,\ldots,b,\ldots)$ where $a$ appears $d_1$ times, $b$ appears $d_2$ times and so on.

From Travkin's proof of Proposition 1 in \cite{T}, it follows that the statement about dimensions is true when we consider the case of complete flags. That is, when $\mu=\mu'=(1^d)$, we let $Y:=Y^{(1^d),(1^d)}$ and we have $\dim Y=\dim Y^{\textbf{t},T,T'}$ where $T,T'$ are standard tableaux. 
That implies that if $\textbf{t}=(\lambda,\theta,\lambda')$, then
\begin{align*}\dim\Zt & = \dim Y^{\textbf{t},T,T'}-\dim(\F_{x,T}\times\F_{x',T'})\\
 & = \dim Y-(\dim\F_x+\dim\F_{x'}) \\
 & = d^2-n_\lambda-n_{\lambda'} \end{align*}
where $n_\lambda=\sum_i (i-1)\lambda_i$.

In the case of partial flags, we know that
$$\dim\Ymm=\dim\Xmm=d^2+d-\frac{1}{2}\sum_i\mu_i^2-\frac{1}{2}\sum_j\mu'^2_j.$$
Further, for $\textbf{t}=(\lambda,\theta,\lambda')$,
\begin{align*}\dim Y^{\textbf{t},T,T'} & =\dim\Zt+\dim(\F_{x,T}\times\F_{x',T'}) \\
 & = \dim\Zt+\dim\F_x^\mu+\dim\F_{x'}^{\mu'} \\
 & = d^2-n_\lambda-n_{\lambda'}+\left(n_\lambda-\frac{1}{2}\left(-d+\sum_i\mu_i^2\right)\right)
 + \left(n_{\lambda'}-\frac{1}{2}\left(-d+\sum_j\mu'^2_j\right)\right) \\
 & = d^2+d-\frac{1}{2}\sum_i\mu_i^2-\frac{1}{2}\sum_j\mu'^2_j \\
 & = \dim\Ymm.\end{align*}
This concludes the proof of the claim. Now, the claim implies that the irreducible components of $\Ymm$ are exactly the closures of the sets $Y^{\textbf{t},T,T'}$. This is enough to prove the proposition because the set $\D^{\mu,\mu'}$ also parametrizes the same irreducible components.
\end{proof}
\begin{defn}\label{mirrsk}The map $\tlo\mapsto(\textbf{t}(\tlo),T(\tlo),T'(\tlo))$ of Proposition \ref{prop1} is called the \emph{mirabolic Robinson-Schensted-Knuth correspondence}.\end{defn}

\section{Combinatorial description of the mirabolic RSK correspondence}
In this section we will describe an algorithm that takes as input a decorated array $\tlo=(\omega,\beta)\in\D^{\mu,\mu'}$ and gives as output a triple $(\textbf{t},T,T')$, with $\textbf{t}=(\lambda,\theta,\lambda')\in\TT$, $T\in\T^\mu_\lambda$, $T'\in\T^{\mu'}_{\lambda'}$. We will then prove that this is the same as the mirabolic RSK correspondence defined geometrically in the previous section. 

\subsection{The Algorithm}
In the algorithm we describe, the row bumping convention is that a new entry $z$ will bump the left-most entry in the row which is \emph{greater or equal} to $z$, as in Section \ref{varrsk}.

\begin{defn}\label{algo}As an input, we have $\tlo=(\omega,\beta)$ where 
$$\omega=\left(\begin{array}{cccc}u(1) & u(2) & \ldots & u(d) \\ w(1) & w(2) & \ldots & w(d) \end{array}\right),\qquad\beta\subset\{1,\ldots,d\}.$$
\begin{itemize}
\item At the beginning, set $T_0=T_0'=\emptyset$ and let $R$ be a single row consisting of the numbers $d+1,\ldots,2d$ 
$$R=\Yboxdim15pt \young({d+1}{d+2}\ldots\yd)$$
\item For $i=1,2,\ldots,d$
\begin{itemize}
\item If $i\in\beta$, let $T_i$ be the tableau obtained by inserting $w(i)$ into the tableau $T_{i-1}$ via row bumping.
\item If $i\notin\beta$, insert $w(i)$ into $R$ by replacing the least element $z\in R$ that is greater or equal to $w(i)$. Then let $T_i$ be the tableau obtained by inserting $z$ into $T_{i-1}$ via row bumping.
\item Construct $T'_{i}$ by the usual recording procedure. That is add a new box to $T'_{i-1}$ in the same place where the row bumping for $T_i$ terminated, and put $u(i)$ in the new box.
\end{itemize}
\item At this point we have $T_d,T'_d$ two semistandard tableaux with $d$ boxes, and the single row $R$. We let $T':=T'_d$ and $\lambda'$ will be its shape (which is also the same shape of $T_d$). 
\item Insert, via row bumping, $R$ into $T_d$, starting from the left. Call $T_{2d}$ the resulting tableau.
\item Let $\nu=(\nu_1,\nu_2,\ldots)$ be the shape of $T_{2d}$, then we have $\theta=(\theta_1,\theta_2,\ldots):=(\nu_2,\nu_3,\ldots)$. That is we define $\theta$ to be the partition obtained from $\nu$ by removing the first part.
\item We let $T:=T_{2d}^{(d)}$, that is $T$ is the tableau obtained from $T_{2d}$ by removing all the boxes with numbers strictly bigger than $d$. We then have $\lambda$ be the shape of $T$.
\item The output is $((\lambda,\theta,\lambda'),T,T')$.  
\end{itemize}
\end{defn}
\begin{thm}\label{thm2}For all $\tlo\in\D^{\mu,\mu'}$, the triple $(\textbf{t}(\tlo),T(\tlo),T'(\tlo))$ of Definition \ref{mirrsk} is the same as the triple obtained by applying the algorithm \ref{algo} to $\tlo$.\end{thm}
The last section of this paper is a proof of this theorem. In Appendix \ref{append} we give an example that illustrates the result and the algorithm.

\subsection{Proof of Theorem \ref{thm2}}
Let $\tlo=(\omega,\beta)\in\D^{\mu,\mu'}$ with $\omega=\left(\begin{smallmatrix} u(1) & \ldots & u(d) \\ w(1) & \ldots & w(d) \end{smallmatrix}\right)$, $\beta\subset\{1,\ldots,d\}$. We want to show that the triple $(\textbf{t}(\tlo),T(\tlo),T'(\tlo))$ of Definition \ref{mirrsk} is the same as what we get applying the algorithm of Definition \ref{algo}. 

Consider $\tlo_+ =(\omega_+,\beta_+)\in\D^{(1^d,\mu,1^d),(1^d,\mu',1^d)}$ defined by
\begin{align*} \omega_+ & =\left(\begin{array}{cccc} u_+(1) & u_+(2) & \ldots & u_+(3d) \\ w_+(1) & w_+(2) & \ldots & w_+(3d) \end{array}\right) \\
\intertext{ where }
u_+(i) & = \left\{\begin{array}{lcl} i & \text{ if } & i\leq d\quad\text{ or if }\quad 2d+1\leq i\leq 3d \\
                                     u(i-d)+d & \text{ if }& d+1\leq i\leq 2d \end{array}\right. \\
w_+(i) & = \left\{\begin{array}{lcl} i+2d & \text{ if } & i\leq d \\
                                     w(i-d)+d & \text{ if } & d+1\leq i \leq 2d \\
                                     i-2d & \text{ if } & 2d+1\leq i\leq 3d \end{array}\right. \\
\beta_+ & = \{i+d|i\in\beta\}                                                                         
\end{align*}
If we look at $\omega_+$ as a matrix, we can visualize it as a block matrix:
$$\omega_+ =\left(\begin{array}{lll} \textbf{0} & \textbf{0} & \textbf{I}_d \\
                         \textbf{0} & \omega & \textbf{0} \\
                         \textbf{I}_d & \textbf{0} & \textbf{0} \end{array}\right) $$
where $\textbf{0}$ is a block of zeros and $\textbf{I}_d$ is the identity $d\times d$ matrix.

Or, as an array,
$$\omega_+ =\left(\begin{array}{ccccccccc} 1 & \ldots & d & u(1)+d & \ldots & u(d)+d & 2d+1 & \ldots & 3d \\
                                          2d+1 & \ldots & 3d & w(1)+d & \ldots & w(d)+d & 1 & \ldots & 3d \end{array}\right).$$                        Then we have a corresponding variety $Y_{\tlo_+}$ which is the closure of $N^*\Omega_{\tlo_+}$. Since $Y_{\tlo_+}$ is irreducible, all the discrete combinatorial data associated to a point $y\in Y_{\tlo_+}$ will agree on an open dense subset. So we let $y=(F,F',v,x,x',v^*)$ be such a general point, where $F$ and $F'$ are partial flags in a $3d$-dimensional vector space $V_+$, $v\in V_+$ and $x+x'+v^*\otimes v=0$.

Choose a basis $\{e_1,e_2,\ldots,e_{3d}\}$ of $V_+$ that satisfies
\begin{align}\notag F_i &=\seq{e_r|w_+(r)\leq i} \\
\label{bas} F'_j &=\seq{e_s|u_+(s)\leq j} \\
\notag v &=\sum_{i\in\beta_+}e_i \end{align}
and let $\{e_i^*\}$ be the dual basis of $V_+^*$.

\begin{defn}For $m\geq 1$, we define inductively two sequences $\{\gamma_m\}$, $\{\delta_m\}$ of subsets of $\{1,\ldots,3d\}$.
\begin{align*}\gamma_1 & :=\{1,\ldots,3d\}\setminus\beta_+  \\
\delta_m & :=\{i\in\gamma_m|\quad\forall j\in\gamma_m,\quad u_+(j)\geq u_+(i)\text{ or }w_+(j)\geq w_+(i)\} \\
\gamma_{m+1}& :=\gamma_m\setminus\delta_m     \end{align*}
\end{defn}
It is easy to see that for all $m=1,\ldots,d$, the set $\delta_m$ consists of the elements $m$, $2d+m$ plus some subset of $\{d+1,\ldots,2d\}$. Also, $\delta_m=\gamma_m=\emptyset$ for all $m>d$.

\begin{lem}\label{claim}For a general conormal vector $(x,x',v^*)$ at the point $(F,F',v)$, we have for $1\leq m\leq d-1$ 
$$(x^*)^m v^*=\sum\limits_{i\in\gamma_{m+1}}\alpha_{m,i}e_i^*$$ 
with $\alpha_{m,m+1},\alpha_{m,m+2d+1}$ both nonzero.\end{lem} 
\begin{proof}Since $x$ preserves the flag $F$, we have that $e_i^*(x(e_j))= 0$ if $w(i)\geq w(j)$. Analogously $e_i^*(x'(e_j))= 0$ if $u(i)\geq u(j)$. Also, we have that $\im(v^*\otimes v)\subset \left<e_{d+1},\ldots,e_{2d}\right>$. Therefore the condition that $x+x'+v^*\otimes v=0$ implies that in the basis $\{e_1,\ldots,e_{3d}\}$ the three operators have the following block matrix form:
\begin{equation}\label{blkfrm} v^*\otimes v = \left(\begin{array}{ccc} \textbf{0} & \textbf{0} & \textbf{0} \\
                         \ast & \ast & \ast \\
                         \textbf{0} & \textbf{0} & \textbf{0} \end{array}\right) \quad 
                         x = \left(\begin{array}{lll} A & \textbf{0} & \textbf{0} \\
                         \textbf{0} & B & \ast \\
                         \textbf{0} & \textbf{0} & C \end{array}\right)\quad x' = \left(\begin{array}{lll} A' & \textbf{0} & \textbf{0} \\
                         \ast & B' & \textbf{0} \\
                         \textbf{0} & \textbf{0} & C' \end{array}\right)\end{equation}                                    
where $A,B,C,A',B',C'$ are strictly upper triangular $d\times d$ matrices and the $\ast$'s are some possibly nonzero matrices depending on $\beta_+$.
They satisfy $A'=-A$ and $C'=-C$.
Now, let
\begin{equation}\label{1jblk}A=\left(\begin{array}{cccc} 0 & a_{1,2} & \ldots & a_{1,d} \\
                                & 0      & \ddots & \vdots \\
                                &        & \ddots & a_{d-1,d} \\
                              0 &        &        &   0      \end{array}\right) \quad
                              C=\left(\begin{array}{cccc} 0 & c_{1,2} & \ldots & c_{1,d} \\
                                & 0      & \ddots & \vdots \\
                                &        & \ddots & c_{d-1,d} \\
                              0 &        &        &   0      \end{array}\right). \end{equation}  
Then, since $F\in\F^{(1^d,\mu,1^d)}$, the set of conormal vectors such that $\rk A = d-1 = \rk C$ (or equivalently such that $a_{1,2},\ldots,a_{d-1,d},c_{1,2},\ldots,c_{d-1,d}$ are all nonzero) is open dense in $N^*\Omega_{\tlo_+}|_{(F,F',v)}$. 

Also, the set of conormal vectors $(x,x',v^*)$ such that $v^*(e_1),v^*(e_{2d+1})$ are both nonzero is open dense. 

Let $J$ be the intersection of these two sets, then $J\subset N^*\Omega_{\tlo_+}|_{(F,F',v)}$ is open dense. From now on we assume that $(x,x',v^*)\in J$ and we are going to prove that the conclusion of the lemma is true.

For $i\in\beta_+$, we have
$$ e_i^*(x(e_i))=e_i^*(x'(e_i))=0 $$
therefore
\begin{align*} 0 & =e_i^*(x(e_i))+e_i^*(x'(e_i)) \\
& = e_i^*((-v^*\otimes v)(e_i)) \\
& = -e_i^*\left(v^*(e_i)\sum_{k\in\beta_+}e_k\right) \\
& = -v^*(e_i).
\end{align*}
So the elements of the basis $\{e_i|i\in\beta_+\}$ are such that $v^*$ vanishes on them, hence 
$$v^*=\sum_{i\notin\beta_+}\alpha_{0,i}e_i^*=\sum_{i\in\gamma_1}\alpha_{0,i}e_i^*$$
for some coefficients $\{\alpha_{0,i}\}$ with $\alpha_{0,1}=v^*(e_1)\neq 0$ and $\alpha_{0,2d+1}=v^*(e_{2d+1})\neq 0$ because $(x,x',v^*)\in J$.

Now, inductively, let us assume that 
$$(x^*)^{m-1}v^*=\sum_{i\in\gamma_m}\alpha_{m-1,i}e_i^*\quad\text{ with }\alpha_{m-1,m}\neq 0\neq\alpha_{m-1,m+2d}$$ 
then
\begin{align*}(x^*)^mv^*(e_j)& = (x^*)^{m-1}v^*(xe_j) \\
& = \sum_{i\in\gamma_m}\alpha_{m-1,i}e_i^*(xe_j) \end{align*}
now if $i\in\gamma_m$, in particular $i\notin\beta_+$, hence
\begin{align}\label{eq3} 0 & = e_i^*(x+x'+v^*\otimes v)e_j \notag \\
& = e_i^*(x(e_j)+x'(e_j)+v^*(e_j)v) \notag \\
& = e_i^*(x(e_j))+e_i^*(x'(e_j)) \notag 
\intertext{therefore}
 e_i^*(x(e_j)) & =-e_i^*(x'(e_j)). \end{align}
The LHS of \eqref{eq3} is nonzero if and only if $w_+(i)<w_+(j)$, while the RHS is nonzero if and only if $u_+(i)<u_+(j)$.

This shows that $(x^*)^mv^*(e_j)=0$ if for all $i\in\gamma_{m-1}$, we have $u_+(i)\geq u_+(j)$ or $w_+(i)\geq w_+(j)$. This is equivalent to
$$(x^*)^mv^*=\sum_{i\in\gamma_{m+1}}\alpha_{m,i}e_i^*$$
for some $\alpha_{m,i}$. Moreover, 
\begin{align*}\alpha_{m,m+1} &= (x^*)^mv^*(e_{m+1}) \\
&= (x^*)^{m-1}v^*(x e_{m+1}) \\
&= (x^*)^{m-1}v^*\left(\sum_{j=1}^m a_{j,m+1}e_j\right) \\
&= \sum_{i\in\gamma_m}\alpha_{m-1,i}e_i^*\left(\sum_{j=1}^m a_{j,m+1}e_j\right)\\
&= \alpha_{m-1,m} a_{m,m+1}\neq 0
\end{align*}
because $j\notin\gamma_m$ for $j<m$ and $(x,x',v^*)\in J.$

Analogously,
\begin{align*}\alpha_{m,m+2d+1} &= (x^*)^mv^*(e_{m+2d+1}) \\
&= (x^*)^{m-1}v^*(x e_{m+2d+1}) \\
&= (x^*)^{m-1}v^*((-x'-v^*\otimes v) e_{m+2d+1}) \\
&= -(x^*)^{m-1}v^*\left(\sum_{j=1}^m(-c_{j,m+1})e_{j+2d}+v^*(e_{m+2d+1})v\right) \\
&= \sum_{i\in\gamma_m}\alpha_{m-1,i}e_i^*\left(\sum_{j=1}^m c_{j,m+1}e_{j+2d}\right)+0\\
&= \alpha_{m-1,m} c_{m,m+1}\neq 0
\end{align*}
because $j\notin\gamma_m$ for $2d<j<m+2d$ and $(x,x',v^*)\in J.$
\end{proof}
                                          
We let $S:=(k[x^*]v^*)^\perp\subset V_+$, that is $S$ is the annihilator of the span of $\{v^*,x^*v^*,(x^*)^2 v^*,\ldots\}\subset V_+^*$. We want to describe the relative position of the partial flags $F\cap S$ and $F'\cap S$.
\begin{defn}We define a new array $\omega'\in\MM{(\mu,1^d)}{(\mu',1^d)}$ by
$$\omega'=\left(\begin{array}{ccc} u'(d+1) & \ldots & u'(3d) \\ w'(d+1) & \ldots & w'(3d) \end{array}\right)$$
where
\begin{align*}u'(i)& =u_+(i) \\
w'(i) & =\left\{\begin{array}{lcl} w_+(i) & \text{ if } & i\in\beta_+ \\
                                   w_+(j), \text{ where }j=\max\{l\in\delta_m|l<i\} & \text{ if } & i\in\delta_m \text{ for some }m.\end{array}\right.
                                   \end{align*}
\end{defn}
Notice that this is well defined because $\beta_+\sqcup\bigsqcup\limits_m \delta_m=\{1,2,\ldots,3d\}$.
\begin{lem}\label{lem3}The relative position of the flags $F\cap S$ and $F'\cap S$ is $\omega'$. \end{lem}
\begin{proof}
Remark that $S=\bigcap_{m=0}^{d-1}\ker\left((x^*)^m v^*\right)$, therefore $ F_d\cap S=0=F'_d\cap S$. This follows from Lemma \ref{claim}, since $(x^*)^mv^*(e_{m+1})$ and $(x^*)^mv^*(e_{m+2d+1})$ are both nonzero.

This implies that the types of the partial flags $F\cap S$ and $F'\cap S$ are respectively $(\mu,1^d)$ and $(\mu',1^d)$. In particular we have $(F\cap S)_{\mu_1+\ldots+\mu_i}=F_{d+\mu_1+\ldots+\mu_i}\cap S$ and analogously for $F'\cap S$.

If we let 
$$r_{ij}(\omega')=\card\{l\in\{d+1,\ldots,3d\}|w'(l)\leq i\text{ and }u'(l)\leq j\},$$ 
to prove the lemma we just need to show that
$$\dim (F_i\cap F'_j\cap S)=r_{ij}(\omega')\quad\forall i,j.$$
We define the set 
$$R_{ij}(\omega_+):=\{l\in\{1,\ldots,3d\}|w_+(l)\leq i\text{ and }u_+(l)\leq j\}$$
then if we let $r_{ij}(\omega_+)=\card R_{ij}(\omega_+)$, we have that $r_{ij}(\omega_+)=\dim(F_i\cap F'_j)$. More precisely, the vectors $\{e_l|l\in R_{ij}(\omega_+)\}$ form a basis for $F_i\cap F'_j$.

Remark that if $m'\leq m$ and $\delta_m\cap R_{ij}(\omega_+)\neq\emptyset$, then $\delta_{m'}\cap R_{ij}(\omega_+)\neq\emptyset$. Hence there exist integers $k_{ij}\geq 0$ defined by the property that $\delta_m\cap R_{ij}(\omega_+)\neq\emptyset$ if and only if $m\leq k_{ij}$. Furthermore, since $\gamma_m=\delta_m\sqcup\delta_{m+1}\sqcup\ldots\sqcup\delta_d$, we have that $\gamma_m\cap R_{ij}(\omega_+)\neq\emptyset$ if and only if $m\leq k_{ij}$. 

This implies that $(x^*)^{m-1}v^*|_{F_i\cap F'_j}\neq 0$ if and only if $m\leq k_{ij}$. Actually, by Lemma \ref{claim} these linear functionals on $F_i\cap F'_j$ are linearly independent for $m=1,\ldots, k_{ij}$. Therefore
\begin{align*}\dim(F_i\cap F'_j\cap S) & = \dim\left(\bigcap_{m=1}^{k_{ij}}\ker(x^*)^{m-1}v^*|_{F_i\cap F'_j}\right) \\
                                       & = \dim(F_i\cap F'_j)-k_{ij} \\
                                       & = r_{ij}(\omega_+)-k_{ij}
\end{align*}
To conclude the proof of the lemma now we need to show that $r_{ij}(\omega_+)-k_{ij}=r_{ij}(\omega')$.
If we now define, in analogy to $R_{ij}(\omega_+)$, $R_{ij}(\omega')$ to be the set such that $r_{ij}(\omega')=\card R_{ij}(\omega')$, we have that
\begin{align*}R_{ij}(\omega_+)& =(R_{ij}(\omega_+)\cap\beta_+)\sqcup\left(\bigsqcup_{m=1}^d R_{ij}(\omega_+)\cap\delta_m\right) \\
R_{ij}(\omega')& =(R_{ij}(\omega')\cap\beta_+)\sqcup\left(\bigsqcup_{m=1}^d R_{ij}(\omega')\cap\delta_m\right).
\end{align*}
By definition of $\omega'$, we have $R_{ij}(\omega')\cap\beta_+=R_{ij}(\omega_+)\cap\beta_+$. 

If $m>k_{ij}$, then $R_{ij}(\omega')\cap\delta_m=R_{ij}(\omega_+)=\emptyset$.

If $m\leq k_{ij}$, then $R_{ij}(\omega')\cap\delta_m=R_{ij}(\omega_+)\cap\delta_m\setminus\{s_m\}$, where $s_m$ is the minimal element of $R_{ij}(\omega_+)\cap\delta_m$.

Therefore $R_{ij}(\omega')=R_{ij}(\omega_+)\setminus\{s_1,\ldots,s_{k_{ij}}\}$ which implies
$$r_{ij}(\omega')=r_{ij}(\omega_+)-k_{ij}.$$  
\end{proof}

\begin{lem}The subspace $S$ of $V_+$ has dimension $2d$ and it is invariant under both $x$ and $x'$.
\end{lem}
\begin{proof}
The dimension claim just follows from the fact that $S=\bigcap_{m=0}^{d-1}\ker\left((x^*)^m v^*\right)$ and that those functionals are linearly independent by Lemma \ref{claim}. Therefore $\dim S=3d-d=2d$. If $z\in S$, then $(x^*)^mv^*(z)=0$ for all $m$, therefore 
$$(x^*)^mv^*(xz)=(x^*)^{m+1}v^*(z)=0,\quad\text{ for all }m$$
and $xz\in S$.
In a similar way, for all $m$,
\begin{align*}(x^*)^mv^*(x'z) &=(x^*)^mv^*((-x-v^*\otimes v)z)\\
&=-(x^*)^{m+1}v^*(z)-v^*(z)(x^*)^mv^*(v) \\
&= 0. 
\end{align*} 
\end{proof}

\begin{defn}\label{bart}Let $\bar{x}=x|_S=-x'|_S$. We then have a map
$$ g:Y_{\tlo_+}\to \F_{\bar{x}}\times\F_{\bar{x}}$$
$$(F,F',v,x,x',v^*)\mapsto (F\cap S,F'\cap S).$$
Since $Y_{\tlo_+}$ is irreducible, $\im g$ lies in an irreducible component of $\F_{\bar{x}}\times\F_{\bar{x}}$. 

So there exist two semistandard tableaux $\bar{T}$, $\bar{T}'$ such that, for all $y\in Y_{\tlo_+}$, $(F\cap S,F'\cap S)\in C_{\bar{x},\bar{T}}\times C_{\bar{x},\bar{T}'}$. 
\end{defn}
In particular, by what we remarked at the beginning of the proof of Lemma \ref{lem3}, we have that $\bar{T}$ has content $(\mu,1^d)$ and $\bar{T}'$ has content$(\mu',1^d)$.

\begin{lem}\label{choose}The map
$$g:Y_{\tlo_+}\to C_{\bar{x},\bar{T}}\times C_{\bar{x},\bar{T}'}$$
is surjective.
\end{lem}
\begin{proof}Let $y=(F,F',v,x,x',v^*)\in Y_{\tlo_+}$, so that $g(y)=(F\cap S,F'\cap S)$. Given $(\bar{F},\bar{F}')\in C_{\bar{x},\bar{T}}\times C_{\bar{x},\bar{T}'}$, define two flags $\hat{F}$, $\hat{F}'$ in $V_+$ as follows
$$\hat{F}_i =\left\{\begin{array}{lcl} F_i & \text{ if } & i\leq d \\
                                      \bar{F}_{i-d}+F_d & \text{ if } & i>d \end{array}\right. $$
and $\hat{F}'$ is defined in the same way, replacing $F'$ and $\bar{F}'$ where necessary. 

Clearly, $x$ preserves the flag $\hat{F}$ and the same is true for $x'$ and $\hat{F}'$. We can then consider the point $\hat{y}=(\hat{F},\hat{F}',v,x,x',v^*)\in Y^{(1^d,\mu,1^d),(1^d,\mu',1^d)}$.
By construction, $\hat{y}$ is such that $\hat{F}\cap S=\bar{F}$ and $\hat{F}'\cap S=\bar{F}'$. 
Consider the maps
$$(\bar{F},\bar{F}')\mapsto (\hat{F},\hat{F}')\mapsto \hat{y}.$$
Let $f$ be the composition of those, then 
$$f:C_{\bar{x},\bar{T}}\times C_{\bar{x},\bar{T}'}\to Y^{(1^d,\mu,1^d),(1^d,\mu',1^d)}.$$
Since $C_{\bar{x},\bar{T}}\times C_{\bar{x},\bar{T}'}$ is irreducible, the image of $f$ lies in an irreducible component of $Y^{(1^d,\mu,1^d),(1^d,\mu',1^d)}$. Notice that $f(F\cap S,F'\cap S)=y\in Y_{\tlo_+}$, hence $\im f\subset Y_{\tlo_+}$. 
Therefore $\hat{y}\in Y_{\tlo_+}$ and $g(\hat{y})=(\bar{F},\bar{F}')$, thus the lemma is proved.
\end{proof}
\begin{rem}For $(F,F',v,x,x',v^*)$ in an open dense subset of $Y_{\tlo_+}$, we know by Lemma \ref{lem3} that $\omega'$ is the relative position of the partial flags $F\cap S$ and $F'\cap S$. Also, by Lemma \ref{choose}, the preimage of the open dense subset of $C_{\bar{x},\bar{T}}\times C_{\bar{x},\bar{T}'}$ for which Theorem \ref{rosso} applies, contains an open dense subset of $Y_{\tlo_+}$. Therefore we have 
$$(\bar{T},\bar{T}')\stackrel{RSK}\longleftrightarrow\omega'.$$ 
\end{rem}
Now, consider the spaces $F_{d+n}$, $F'_{d+n'}$ where $n$, $n'$ are the number of parts of $\mu$ and $\mu'$ respectively, i.e. $\mu=(\mu_1,\ldots,\mu_n)$, $\mu'=(\mu'_1,\ldots,\mu'_{n'})$. By \eqref{blkfrm}, they are invariant under both operators $x$ and $x'$, therefore the same is true for $V:=F_{d+n}\cap F'_{d+n'}$. Notice that in the basis of \eqref{bas}, we have that $V=\seq{e_{d+1},\ldots,e_{2d}}$.

Consider the flags $F\cap V$ and $F'\cap V$. It is clear that the relative position $M(F\cap V,F'\cap V)=\omega$, and that $\bar{y}=(F\cap V, F'\cap V, v, x|_V,x'|_V,v^*|_V)\in Y_{\tlo}$. 

Applying the mirabolic RSK correspondence of Definition \ref{mirrsk} to $Y_{\tlo}$ we get 
\begin{equation}\label{5tuple}(\textbf{t}(\tlo),T(\tlo),T'(\tlo)),\quad\text{ with }\quad\textbf{t}(\tlo)=(\lambda(\tlo),\theta(\tlo),\lambda'(\tlo)).\end{equation} 

Thus we have $F\cap V\in C_{x|_V,T(\tlo)}$, $F'\cap V\in C_{x'|_V,T'(\tlo)}$ and $\theta(\tlo)$ is the Jordan type of $x|_{V/k[x]v}$.
\begin{lem}\label{rmv}The semistandard tableau $T(\tlo)$ (resp. $T'(\tlo)$) is obtained from the tableau $\bar{T}$ (resp. $\bar{T}'$) of Definition \ref{bart} by removing all boxes with numbers $n+1,\ldots,n+d$ (resp. $n'+1,\ldots,n'+d$).\end{lem}
\begin{proof}By symmetry, it is enough to prove the case of $T(\tlo)$. The tableau $\bar{T}$ is defined by the condition that $F\cap S\in C_{x|_V,\bar{T}}$. If we let $T^{(n)}$ be the tableau obtained by removing from $\bar{T}$ all numbers greater than $n$, we have
$$F\cap S\cap F_{d+n}\in C_{x|_{S\cap F_{d+n}},T^{(n)}}.$$

By the remark at the beginning of the proof of Lemma \ref{lem3} and by the the definition of $V$, the spaces $V$ and $S\cap F_{d+n}$ are both complementary to $F_d$ inside $F_{d+n}$. So they can both be identified with the image of the map \begin{equation}\label{ident}F_{d+n}\twoheadrightarrow F_{d+n}/F_d.\end{equation}
Notice that under this map
\begin{align*}F\cap V & \mapsto (F\cap F_{d+n})/F_d \\
F\cap S\cap F_{d+n} & \mapsto (F\cap F_{d+n})/F_d \end{align*}
and both operators $x|_V$ and $x|_{S\cap F_{d+n}}$ get identified via \eqref{ident} with $x|_{F_{d+n}/F_d}$.
Therefore it follows that
\begin{align*}(F\cap F_{d+n})/F_d & \in C_{x|_{F_{d+n}/F_d},T(\tlo)} \\
(F\cap F_{d+n})/F_d & \in C_{x|_{F_{d+n}/F_d},T^{(n)}}. \end{align*}
By Lemma \ref{choose} the set of all $F\cap F_{d+n}$, for varying $y\in Y_{\tlo_+}$, covers all points in these irreducible components. Therefore they must be equal, i.e. 
$$ C_{x|_{F_{d+n}/F_d},T(\tlo)}=C_{x|_{F_{d+n}/F_d},T^{(n)}}$$
which implies that $T(\tlo)=T^{(n)}$.
\end{proof}

Suppose that applying the algorithm \ref{algo} to $\tlo$ we obtain $(\lambda^c,\theta^c,(\lambda')^c,T^c,(T')^c)$. We want to show that this coincides with the quintuple $(\lambda(\tlo),\theta(\tlo),\lambda'(\tlo),T(\tlo),T'(\tlo))$ of \eqref{5tuple}. 

Remark that $i+d\in\delta_m$ if and only if at the $i$-th step of the algorithm, the number $w(i)$ is being inserted in the $m$-th position of the row $R$. In this case, $w'(i)$ is the number bumped from $R$ and inserted in $T_i$. 

Therefore, if we apply the RSK correspondence from Section \ref{varrsk} to $\omega'$ we get a pair of tableaux $(T(\omega'),T'(\omega'))$ that satisfy the following:
\begin{itemize}
\item the tableau $T_{2d}$ from the algorithm is the same as $T(\omega')$; 
\item the tableau $T'_d$ is obtained from $T'(\omega')$ by removing all numbers strictly greater than $n'$.
\end{itemize}  
We also know that $\omega'\stackrel{RSK}\longleftrightarrow (\bar{T},\bar{T}')$, therefore $\bar{T}=T(\omega')$ and $\bar{T}'=T'(\omega')$. 

By Lemma \ref{rmv}, this implies that both $T^c$ and $T(\tlo)$ are obtained from $\bar{T}=T(\omega')=T_{2d}$ by removing the last $d$ numbers, so $T^c=T(\tlo)$. 

Again by Lemma \ref{rmv}, $(T')^c=T'_d$ and $T'(\tlo)$ are both obtained from $\bar{T}'=T'(\omega')$ by removing all numbers greater than $n'$, so $(T')^c=T'(\tlo)$. 

It also follows immediately that $\lambda^c=\lambda(\tlo)$ and $(\lambda')^c=\lambda'(\tlo)$ since those are respectively the shape of $T(\tlo)$ and of $T'(\tlo)$.

The only thing left to prove is that $\theta^c=\theta(\tlo)$, which will follow from the next Lemma.
\begin{lem}If we let $\nu$ be the shape of the tableau $\bar{T}$, then $\theta:=\theta(\tlo)$ is obtained from $\nu$ by removing the first part of the partition.
That is $\theta=(\theta_1,\theta_2,\ldots)=(\nu_2,\nu_3,\ldots)$. \end{lem}
\begin{proof}[Proof of Lemma]
From Definition \ref{bart}, the shape of $\bar{T}$ is the Jordan type of $\bar{x}=x|_S$. On the other hand, $\theta$ is the type of $x|_{V/k[x]v}$.

Consider the space $D:=(F_d+F'_d)\cap S+k[x]v$. Since $v\in S$ by Lemma \ref{claim} and since $S$ is $x$-invariant, we have $k[x]v\subset S$. Therefore $D=(F_d+F'_d+k[x]v)\cap S$. Now, $F_d$ is $x$-invariant by \eqref{blkfrm}, and if $z\in F'_d$, then $x(z)=-x'(z)-v^*(z)v\in F'_d+k[x]v$. So $F'_d+k[x]v$ is also $x$-invariant. It follows that $D$ is invariant under $x$. 
\begin{claim}In Jordan normal form, the nilpotent operator $x|_D$ has a single block.\end{claim}
\begin{proof}[Proof of Claim]We can assume that the matrices $A$ and $C$ of \eqref{1jblk} have rank $n-1$. 
Therefore $x|_{F_d}$, which is represented by the matrix $A$, has a single Jordan block. In the same way, $-C$ represents $x'|_{F'_d}$ which also has a single Jordan block.

Now, given the basis $\{e_i|i=1,\ldots,3d\}$ of $V_+$ defined in \eqref{bas}, we have that 
$$F_d+F'_d=\seq{e_1,\ldots,e_d,e_{2d+1},\ldots,e_{3d}}$$
hence, by Lemma \ref{claim} for $m=0,\ldots,d-1$, the linear functionals
$$(x^*)^m v^*=\sum\limits_{i\in\gamma_{m+1}}\alpha_{m,i}e_i^*$$
are linearly independent on $F_d+F'_d$. It follows that 
\begin{align*}\dim (F_d+F'_d)\cap S & =2d-d \\
                                    &  =d \\
              \dim D & = \dim (F_d+F'_d)\cap S +k[x]v \\     
               & = d+\dim k[x]v \end{align*}
since $(F_d+F'_d)\cap S\cap k[x]v=0$. (This is because $k[x]v\subset\seq{e_{d+1},\ldots,e_{2d}}$ by \eqref{blkfrm}).

Now, let $z=\sum_{i=1}^d z_ie_i\in F_d$, with $z_d\neq 0$, and let $z'=\sum_{i=2d+1}^{3d}z_ie_i\in F'_d$, with $z_{3d}\neq 0$.
We want to show that we can choose the $z_i's$ in such a way that $z+z'\in S$.
Consider the equation
\begin{align*}0 &= (x^*)^{d-1}v^*(z+z') \\
                &= \left(\sum_{i\in\gamma_d}\alpha_{d-1,i}e_i^*\right)\left(\sum_{\substack{j=1,\ldots,d \\ 2d+1,\ldots,3d}}z_je_j\right)\\
                &= \alpha_{d-1,d}z_d+\alpha_{d-1,3d}z_{3d}
                \end{align*}
Since, by Lemma \ref{claim}, $\alpha_{d-1,d}$ and $\alpha_{d-1,3d}$ are both nonzero, we can find nonzero $z_d,z_{3d}$ such that the equation holds.
We find
\begin{align*}0 &= (x^*)^{d-2}v^*(z+z') \\
                &= \left(\sum_{i\in\gamma_{d-1}}\alpha_{d-2,i}e_i^*\right)\left(\sum_{\substack{j=1,\ldots,d \\ 2d+1,\ldots,3d}}z_je_j\right)\\
                &= \alpha_{d-2,d-1}z_{d-1}+\alpha_{d-2,d}z_d+\alpha_{d-2,3d-1}z_{3d-1}+\alpha_{d-2,3d}z_{3d}.
                \end{align*}
Since $\alpha_{d-2,d-1}$ and $\alpha_{d-2,3d-1}$ are both nonzero, we can choose $z_{d-1},z_{3d-1}$ so that the equation holds.
Iterating this procedure, we find $z,z'$ such that $(x^*)^mv^*(z+z')=0$ for all $0\leq m\leq d-1$, hence $z+z'\in (F_d+F'_d)\cap S$.

Remark that, since $z_d\neq 0$, and since $x$ acts as the matrix $A$ from \eqref{1jblk} on $F_d$, $F_d=\seq{z,\ldots,x^{d-1}z}$.
In the same way, 
we have $F'_d=\seq{z',\ldots,(x')^{d-1}z'}$.               

We are now going to prove that $z+z'$ is a cyclic vector for $x$ on $D$.

We have
\begin{align}\label{comp} x(z+z') & = x(z)+x(z') \notag \\
                       & = x(z)-x'(z')-v^*(z')v \notag \\
             x^2(z+z') & = x ( x(z)-x'(z')-v^*(z')v) \notag \\
                       & = x^2(z) -v^*(z')x(v)+(-x'-v^*\otimes v)(-x'(z')) \\
                       & = x^2(z)+(x')^2(z')+v^*(x'z')v-v^*(z')x(v) \notag \\
               \ldots  & \ldots \notag \\
             x^d(z+z') & = x^d(z)+(x')^d(z')+(-1)^d v^*((x')^{d-1}z')v+\ldots -v^*(z')x^{d-1}v. \notag         
                       \end{align}

Remark that, since $z_d\neq 0$, and since $x$ acts as the matrix $A$ from \eqref{1jblk} on $F_d$, $F_d=\seq{z,\ldots,x^{d-1}z}$.
In the same way, we have $F'_d=\seq{z',\ldots,(x')^{d-1}z'}$.               
Also, notice that 
\begin{align*}v^*((x')^{d-1}z')&= v^*((-1)^{d-1} c_{1,2}\cdots c_{d-1,d}e_{2d+1}) \\
                               &= (-1)^{d-1} c_{1,2}\cdots c_{d-1,d}\alpha_{0,2d+1}\neq 0
                               \end{align*}
Since $x^d(z)=(x')^d(z')=0$, we have $x^d(z+z')\in k[x]v$ and has a nonzero coefficient in $v$.
Therefore, it follows from the computation \eqref{comp} that $x^m(z+z')$ are linearly independent for $m=0,\ldots,d$. 

Moreover, the elements 
$$x^m(z+z')\quad\text{ with }\quad d\leq m\leq d+(\dim k[x]v-1),$$  
span $k[x]v$. 
In conclusion, the set $\{x^m(z+z')|m=0,\ldots, d+(\dim k[x]v-1)\}$ spans $D=(F_d+F'_d)\cap S+k[x]v$. This means that $z+z'$ is a cyclic vector, hence $x|_D$ has a single block in Jordan normal form.                    
\end{proof}

Now, the identification of \eqref{ident} gives us an isomorphism of $x$-modules 
$$\alpha:V\stackrel{\simeq}\to S\cap F_{d+n}.$$
Remark that $D\cap V=k[x]v$, and that $D+(S\cap F_{d+n})=S$. Also
\begin{align*}D\cap\alpha(V) &= D\cap(F_{d+n}\cap S) \\
&=(F_d+F'_d+k[x]v)\cap F_{d+n}\cap S \\
&=(F_d+k[x]v)\cap S \\
&=F_d\cap S +k[x]v \\
&=k[x]v
\end{align*} 
We have then isomorphisms of $x$-modules
\begin{align*} V/k[x]v  & = V/(V\cap D)  \\
                        & \simeq (D+V)/D \\
                        & \simeq (D+\alpha(V))/D \\
                        & = S/D.
                        \end{align*}  
Hence, $x|_{V/k[x]v}=x|_{S/D}$. So $\theta$ is also the Jordan type of $x|_{S/D}$. 

We know that $\dim D\geq d$, $\dim S=2d$ and $x|_D$ is a single Jordan block. It follows that $\theta$, which is the Jordan type of $x|_{S/D}$, is obtained from the one of $x|_S$ by removing the maximal part of the partition. This concludes the proof of the Lemma and consequently of the Theorem.
\end{proof}

\appendix
\section{Example of the mirabolic RSK correspondence}\label{append}
Let $V\simeq k^7$, and let a basis of $V$ be $\{u_1,u_2,\ldots,u_7\}$. We consider the nilpotents $x$, $x'$, expressed as matrices in the basis $\{u_i\}$.
$$ x=\begin{pmatrix} 0 & 0 & 0 & 0 & 0 & 0 & 0 \\ 
0 & 0 & 1 & 0 & 0 & 0 & 0 \\
0 & 0 & 0 & 1 & 0 & 0 & 0 \\
0 & 0 & 0 & 0 & 1 & 0 & 0 \\
0 & 0 & 0 & 0 & 0 & 0 & 0 \\
0 & 0 & 0 & 0 & 0 & 0 & 1 \\
0 & 0 & 0 & 0 & 0 & 0 & 0 \\
\end{pmatrix};\quad x'=\begin{pmatrix} 0 & 0 & 0 & 0 & 0 & 0 & 0 \\ 
0 & 0 & 0 & 0 & 0 & 0 & 0 \\
0 & 0 & 0 & -1 & 0 & 0 & 0 \\
0 & 0 & 0 & 0 & -1 & 0 & 0 \\
0 & 0 & 0 & 0 & 0 & 0 & 0 \\
0 & 0 & 0 & 0 & 0 & 0 & -1 \\
0 & 0 & 0 & 0 & 0 & 0 & 0 \\
\end{pmatrix}. $$
Then we have that the Jordan type of $x$ is $\lambda=(4,2,1)$ and the type of $x'$ is $\lambda'=(3,2,1,1)$. If we let $v=u_2$, $v^*=-u_3^*$, we have indeed
$$ x+x'+v^*\otimes v=0.$$
Now, $x(v)=0$, therefore $k[x]v=\seq{v}$ and $V/k[x]v\simeq\seq{u_i|i\neq 2}$. We then have that the type of $x|_{V/k[x]v}$ is $\theta=(3,2,1)$.

Let us define the flag $F$ by
\begin{align*}F_1 &= \seq{u_2,u_6} \\
              F_2 &= \seq{u_2,u_6,u_3,u_1+u_7} \\
              F_3 &= \seq{u_2,u_6,u_3,u_1+u_7,u_4} \\
              F_4 &= V. \end{align*}
Then $F\in\F^\mu_x$ for $\mu=(2,2,1,2)$. We also define $F'\in\F^{\mu'}_{x'}$, with $\mu'=(2,2,3)$, by
\begin{align*}F'_1 &= \seq{u_1,u_3} \\
              F'_2 &= \seq{u_1,u_3,u_4,u_6} \\
              F'_3 &= V. \end{align*}
The semistandard tableaux associated to $F$ and $F'$ are respectively
$$T=\young(1234,12,4) \qquad T'=\young(123,13,2,3)$$
In which $\gl_7$-orbit does the point $(F,F',v)$ lie?
The relative position of $F$ and $F'$ is $\omega\in\MM{(2,2,1,2)}{(2,2,3)}$ which we can see as a matrix or as an array
$$\omega=\begin{pmatrix} 0 & 1 & 1 \\ 1 & 0 & 1 \\ 0 & 1 & 0 \\ 1 & 0 & 1 \end{pmatrix}=
\begin{pmatrix} 1 & 1 & 2 & 2 & 3 & 3 & 3 \\ 4 & 2 & 3 & 1 & 4 & 2 & 1 \end{pmatrix} $$
and since $v\in F_1\setminus (F_1\cap F'_2)$, we have that $\beta=\{4,7\}$. 
So for $\tlo=(\omega,\beta)$, we have $y=(F,F',v,x,x',v^*)\in N^*\Omega_{\omega,\beta}$.
Now, by the mirabolic RSK correspondence of Definition \ref{mirrsk}, we have
$$(\omega,\beta)\to(\lambda,\theta,\lambda',T,T').$$
Let us verify that this is indeed the result we obtain when we apply the algorithm \ref{algo}.
Our input is
$$(\omega,\beta)=\left( \begin{pmatrix} 1 & 1 & 2 & 2 & 3 & 3 & 3 \\ 4 & 2 & 3 & 1 & 4 & 2 & 1 \end{pmatrix}, \{4,7\}\right).$$
To start, we set $T_0=T'_0=\emptyset$, $R=\young(89\yy\yya\yyb\yyc\yyd)$
\begin{itemize}
\item $1\notin\beta$.
$$R=\young(49\yy\yya\yyb\yyc\yyd)\qquad T_1=\young(8)\qquad T'_1=\young(1)$$
\item $2\notin\beta$.
$$R=\young(29\yy\yya\yyb\yyc\yyd)\qquad T_2=\young(4,8)\qquad T'_2=\young(1,1)$$
\item $3\notin\beta$.
$$R=\young(23\yy\yya\yyb\yyc\yyd)\qquad T_3=\young(49,8)\qquad T'_3=\young(12,1)$$
\item $4\in\beta$.
$$R=\young(23\yy\yya\yyb\yyc\yyd)\qquad T_4=\young(19,4,8)\qquad T'_4=\young(12,1,2)$$
\item $5\notin\beta$.
$$R=\young(234\yya\yyb\yyc\yyd)\qquad T_5=\young(19\yy,4,8)\qquad T'_5=\young(123,1,2)$$
\item $6\notin\beta$.
$$R=\young(234\yya\yyb\yyc\yyd)\qquad T_6=\young(12\yy,49,8)\qquad T'_6=\young(123,13,2)$$
\item $7\in\beta$.
$$R=\young(234\yya\yyb\yyc\yyd)\qquad T_7=\young(12\yy,19,4,8)\qquad T'_7=\young(123,13,2,3)$$
\item $T'=T'_7=\young(123,13,2,3)$ which agrees with what we had before.
\item Insert $R$ into $T_7$, get
$$T_{14}=\young(1234\yya\yyb\yyc\yyd,12\yy,49,8)$$
\item The shape of $T_{14}$ is $\nu=(8,3,2,1)$, so $\theta=(3,2,1)$ as we wanted. 
\item Removing all numbers greater than $7$ from $T_{14}$ we get
$$T=T^{(7)}_{14}=\young(1234,12,4).$$
\end{itemize}

\subsection*{Acknowledgements}
The author would like to thank Victor Ginzburg for posing the problem and for his help and suggestions, and Jonah Blasiak for useful discussions on the subject. In addition, he would like to thank Joel Kamnitzer for pointing out the result in \cite{H} and Anthony Henderson for pointing out the reference \cite{Sp1}. He also thanks Sergey Fomin for some useful feedback. Finally, he is  grateful to the University of Chicago for support.

\end{document}